\documentclass[12pt,twoside]{article}
\title{Gorenstein $n$-$\X$-injective and  $n$-$\X$-flat modules with respect to a special finitely presented module}
\date{}
\author{}
\usepackage{graphicx,fancyhdr,fancybox,rotating,xcolor,appendix}
\usepackage{ulem,hhline,dsfont,pifont,multicol,lmodern}
\usepackage{amsthm,amsbsy,geometry,times,pst-node}
\usepackage{amsmath,tikz-cd}

\usepackage{lipsum}
\usepackage{euscript}
\usepackage{textcomp}
\usepackage{pdfpages}
\usepackage{stmaryrd}
\usepackage[all]{xy}
\usepackage{amsfonts}
\usepackage{amsmath}
\usepackage{amssymb}
\usepackage{latexsym}
\usepackage{mathrsfs}

\newtheorem{thm}{Theorem}[section]
 \newtheorem{cor}[thm]{Corollary}
 \newtheorem{lem}[thm]{Lemma}
 \newtheorem{prop}[thm]{Proposition}
 \newtheorem{Def}[thm]{Definition}
\newtheorem{rem}[thm]{Remark}
 
 \newtheorem{exs}[thm]{Examples}

\newcommand{\X}{\rm \mathscr{X}}

\def\XIn{{\mathscr{XI}_n}}
\def\XFn{{\mathscr{XF}_n}}

\def\Ext{{\rm Ext}}
\def\Tor{{\rm Tor}}

\newcommand{\F}{{\cal F}}

\def\fd{{\rm fd}}
\def\id{{\rm id}}

\begin{document}

\thispagestyle{empty}

\maketitle \vspace*{-1.5cm}
\begin{center}{\large\bf Mostafa Amini$^{1,a}$, Arij Benkhadra$^{2,b}$ and  Driss   Bennis$^{2,c}$}
\bigskip

\small{1. Department of Mathematics, Faculty of Sciences, Payame Noor University, Tehran, Iran.\\
2. Department of Mathematics, Faculty of Sciences, Mohammed V University in Rabat, Rabat, Morocco.\\
$\mathbf{a.}$ amini.pnu1356@gmail.com \\
$\mathbf{b.}$ benkhadra.arij@gmail.com\\
$\mathbf{c.}$  driss.bennis@um5.ac.ma; driss$\_$bennis@hotmail.com \\}

\small{$\mathbf{}$  
}
\end{center}

\bigskip

\noindent{\large\bf Abstract.} Let $R$ be a  ring, $\X$ a class of $R$-modules and $n\geq1$ an integer. In this paper, via
special finitely presented modules, we introduce the concepts of Gorenstein $n$-$\X$-injective and $n$-$\X$-flat modules. And aside, we obtain some equivalent properties of these modules on $n$-$\X$-coherent rings. Then, we investigate the relations
among Gorenstein $n$-$\X$-injective, $n$-$\X$-flat, injective and flat modules on $\X$-$FC$-rings (i.e.,  self $n$-$\X$-injective and   $n$-$\X$-coherent rings). Several known results are generalized to this new context.\bigskip

\medskip

\bigskip\bigskip
%

\section{Introduction}
 In 1995, Enochs et al. introduced the concept 
of Gorenstein injective and Gorenstein flat modules. Then, these modules have become a vigorously
active area of research. For background on Gorenstein  
modules, we refer the reader to \cite{ EJ, EO,  HM}. In 2012, Gao and Wang introduced and studied in \cite{Z.G} Gorenstein  FP-injective modules. They established various     homological properties of Gorenstein  FP-injective modules mainly over a coherent ring (for more details, see \cite{LD}).

Recall that   coherent rings were first appeared in Chase's paper \cite{Chase} without being mentioned by name. 
The term coherent was first used by Bourbaki in \cite{Bou}. Then,   $n$-coherent rings  were introduced by Costa in \cite{Costa}. Let $n$ be a non-negative integer.  An $R$-module   $M$ is said to be {\it $n$-presented} if there is an exact sequence 
 $ F_{n}\rightarrow F_{n-1}\rightarrow\dots\rightarrow F_1\rightarrow F_0\rightarrow
M\rightarrow$ of   $R$-modules, where each $F_i$ is finitely generated free. And a ring $R$ is called left {\it $n$-coherent} if every
$n$-presented  $R$-module is $(n + 1)$-presented. Thus,  for $n=1$,  left $n$-coherent rings  are nothing but   coherent rings (see \cite{Costa, DKM, Su}).
 Chen and Ding in \cite{CD}  introduced, by using $n$-presented modules,     $n$-FP-injective and $n$-flat modules.  Bennis in \cite{D.Benn} introduced    $n$-$\X$-injective and $n$-$\X$-flat modules and $n$-$\X$-coherent rings for any class $\X$ of $R$-modules.
Then, in 2018, Zhao et al. in \cite{NG} introduced   $n$-FP-gr-injective graded  modules, $n$-gr-flat graded right modules and   $n$-gr-coherent graded rings on a class of  graded $R$-modules. Moreover, they defined   special finitely
presented graded  modules via projective resolutions of $n$-presented graded
left modules, where if $U$ is $n$-presented  graded  module, then in the exact sequence 
 $ F_{n}\rightarrow F_{n-1}\rightarrow\dots\rightarrow F_1\rightarrow F_0\rightarrow
U\rightarrow 0$, $K_{n-1}={\rm Im}(F_{n-1}\rightarrow F_{n-2})$ is called a special finitely presented module. In this paper, we unify and extend various homological notions,  including the ones cited above, to a  more general context.  Namely, we define    special finitely
presented modules via projective resolutions of $n$-presented   modules in a given  $\X$ of $R$-modules. Then, we introduce and study  Gorenstein $n$-$\X$-injective and $n$-$\X$-flat modules with respect to special finitely presented modules. 

The paper is organized as follows:

In Section 2,  some fundamental concepts and some preliminary results are stated.

 In Section 3, we give some characterizations of $n$-$\X$-injective and $n$-$\X$-flat modules.
 
  In Section 4, we introduce  the notions  of  Gorenstein $n$-$\X$-injective and $n$-$\X$-flat modules. We generalize some results of \cite{NG} to the context of $n$-$\X$-injective and $n$-$\X$-flat modules as well as some results  of \cite{Z.G}  to the context of Gorenstein $n$-$\X$-injective modules. Then, we give some   characterizations of Gorenstein $n$-$\X$-injective and $\X$-flat modules on $n$-$\X$-coherent rings.

 In Section 5, we introduce and investigate $\X$-$FC$ rings (i.e.,  self $n$-$\X$-injective and   $n$-$\X$-coherent rings) whose every  module is 
Gorenstein $n$-$\X$-injective and every Gorenstein $n$-$\X$-injective right module is Gorenstein $n$-$\X$-flat. Furthermore, examples are given which 
show that the Gorenstein $m$-$\X$-injectivity (resp., the $m$-$\X$-flatness) does not imply, in general, the Gorenstein $n$-$\X$-injectivity (resp., the $n$-$\X$-flatness) for any $m>n$.

\section{Preliminaries} 
Throughout this paper $R$ will be an associative (non necessarily commutative) ring with identity, and all modules will be unital left $R$-modules (unless specified otherwise).

 In this section,
 some fundamental concepts and notations are stated.
 
  Let $n$ be a non-negative integer,
$M$ an $R$-module and $\X$ a class of   $R$-modules. Then, 
 $M$  is said to be {\it Gorenstein injective} (resp., {\it Gorenstein flat})
\cite{ EJ, EO} if there
is an exact sequence $\cdots\rightarrow I_1\rightarrow I_{0}\rightarrow I^0\rightarrow
I^1\rightarrow\cdots$ of injective (resp., flat )   modules
with $M= {\rm ker}(I^0\rightarrow
I^1)$ such that ${\rm Hom}(U,-)$ (resp.,  $U\otimes_{R}-$) leaves the sequence exact whenever
$U$ is an injective  left (resp.,  right) module. The Gorenstein projective modules
are defined dually. Recall also that $M$ is said to be {\it $n$-FP-injective} \cite{CD} if ${\rm Ext}_{R}^{n}(U,M)=0$
 for any $n$-presented   $R$-module $U$. In case $n=1$, $n$-FP-injective modules are nothing but the well-known  FP-injective modules. A right module
$N$ is called {\it $n$-flat} if ${\rm Tor}_{n}^{R}(N,U)=0$
 for any $n$-presented   $R$-module $U$. Also, 
  $M$ is said to be {\it Gorenstein FP-injective} \cite{Z.G} if there is an exact sequence ${\mathbf{E}}= \cdots\rightarrow E_1\rightarrow E_{0}\rightarrow E^0\rightarrow
E^1\rightarrow\cdots$ with $M=\ker(E^0\rightarrow E^1)$ such that ${\rm Hom}_{R}(P,{\mathbf{E}})$ is an  exact sequence whenever $P$  is finitely presented with ${\rm pd}_R(P)<\infty.$ A graded   $R$-module
$M$ is called {\it $n$-FP-gr-injective} \cite{NG} if ${\rm EXT}_{R}^{n}(N,M)=0$ for any  $n$-presented graded
  $R$-module $N$. A graded right $R$-module
$M$ is called {\it $n$-gr-flat} \cite{NG} if ${\rm Tor}_{R}^{n}(M,N)=0$ for any $n$-presented graded  
$R$-module $N$ (for more details about graded modules, see \cite{  Z.P, LM, CNA}).

From now on,  $\X_k$ is a non empty   class of $k$-presented  $R$-module  in a given class  $\X$ for any integer $k\geq 0$. 

An $R$-module $M$ is said to be {\it $n$-$\X$-injective} \cite{D.Benn} if ${\rm Ext}_{R}^{n}(U,M)=0$ for any $U\in\X_n$.  A right $R$-module
$N$ is called {\it $n$-$\X$-flat} \cite{D.Benn} if ${\rm Tor}_{R}^{n}(N,U)=0$ for any $U\in\X_n$.  We use $\XIn$ (resp.,  $\XFn$) to denote  the class of all $n$-$\X$-injective left $R$-modules (resp.,  $n$-$\X$-flat right $R$-modules). 

A ring $R$ is called left {\it $n$-$\X$-coherent} if every $n$-presented   $R$-module  in $\X$ is $(n+1)$-presented. It is clear that when $n=0$ (resp.,  $n=1$) and $\X$ is a class of all cyclic $R$-modules, then $R$ is Noetherian (resp.,  coherent).

\section{$n$-$\X$-injective, $n$-$\X$-flat and special $\X$-presented modules}
In this section,   we state a relative version of \cite[Definition 3.1]{NG} and provide several characterizations of $n$-$\X$-injective  and $n$-$\X$-flat modules.  

Let us introduce the following notions. 
Let $n \geq 0$ be an integer and $U\in\X_n$ for a  class $\X$ of   $R$-modules. Then,  an  exact sequence of the form 
$$ F_n\longrightarrow F_{n-1}\longrightarrow\cdots\longrightarrow F_1\longrightarrow F_0\longrightarrow U\longrightarrow 0,$$
where each $F_i$ is a finitely generated free  $R$-module, exists.  Let $K_{n-1}={\rm Im}(F_{n-1}\rightarrow F_{n-2})$ and $K_{n}={\rm Im}(F_{n}\rightarrow F_{n-1})$. The short exact sequence $0\rightarrow K_n\rightarrow F_{n-1}\rightarrow K_{n-1}\rightarrow 0$ is called a special short exact sequence of $U$. It is clear that $K_n$ and $K_{n-1}$ are finitely generated and finitely presented, respectively. We call    $K_n$ a special $\X$-generated $R$-module and $K_{n-1}$  a  special $\X$-presented  $R$-module.

Also, a short exact sequence $0\rightarrow A\rightarrow B\rightarrow C\rightarrow 0$ of  $R$-modules is called special $\X$-pure, if for every special $\X$-presented $K_{n-1}$,  there exists the following exact sequence:
$$0\rightarrow {\rm Hom}_{R}(K_{n-1},A)\rightarrow {\rm Hom}_{R}(K_{n-1},B)\rightarrow {\rm Hom}_{R}(K_{n-1},C)\rightarrow 0.$$
 The $R$-module $A$ is said to be a special $\X$-pure in $B$. Also, the exact sequence  $0\rightarrow C^{*}\rightarrow B^{*}\rightarrow A^{*}\rightarrow 0$ is called a split special exact sequence. If $M$ is an $n$-$\X$-injective  (resp.,  a flat) $R$-module, then 
${\rm Ext}_{R}^{1}(K_{n-1}, M)\cong {\rm Ext}_{R}^{n}(U, M)=0$ (resp.,  ${\rm Tor}_{1}^{R}(M, K_{n-1})\cong {\rm Tor}_{n}^{R}(M,U)=0$) for any $U\in\X_n$.\\
 In particular, if $\X$ is a class of graded  $R$-modules, then every special $\X$-generated and every special $\X$-presented module are special finitely generated and special finitely presented graded  $R$-modules, respectively. Also, every $n$-$\X$-injective  $R$-module and every $n$-$\X$-flat right $R$-module are  $n$-FP-gr-injective and $n$-gr-flat, respectively; see \cite{NG}.

\begin{prop}\label{3.8nn} 
Let  $\X$ be a class of  $R$-modules and $M$ an  $R$-module. Then, the following statements are equivalent:
\begin{enumerate}
\item [\rm (1)]
$M$ is $n$-$\X$-injective;
\item [\rm (2)]
Every short exact sequence $0\rightarrow M\rightarrow A\rightarrow C\rightarrow 0$ is special $\X$-pure;
\item [\rm (3)]
$M$ is special $\X$-pure in any injective  $R$-module containing it;
\item [\rm (4)]
$M$ is special $\X$-pure in $E(M)$.
\end{enumerate}
\end{prop}
\begin{proof}
{$(1)\Longrightarrow (2)$ Let $U\in\X_n$. 
Then, there exists the following exact sequence 
$$ F_n\longrightarrow F_{n-1}\longrightarrow\cdots\longrightarrow F_1\longrightarrow F_0\longrightarrow U\longrightarrow 0,$$
where each $F_i$ is a finitely generated free  $R$-module. If $K_{n-1}={\rm Im}(F_{n-1}\rightarrow F_{n-2})$ is special $\X$-presented, then  ${\rm Ext}_{R}^{1}(K_{n-1}, M)\cong{\rm Ext}_{R}^{n}(U, M)=0$ since $M$ is $n$-$\X$-injective. Hence (2) follows.

$(2)\Longrightarrow (3)$ Consider the canonical short exact sequence $0\rightarrow M\rightarrow E\rightarrow \frac{E}{M}\rightarrow 0$, where  $E$ is an injective $R$-module containing $M$. So by (2), $M$ is special $\X$-pure in $E$.

$(3)\Longrightarrow (4)$ is trivial.

$(4)\Longrightarrow (1)$ Assume that $U\in\X_n$ and $K_{n-1}$ is special $\X$-presented. By (4), the short exact sequence $0\rightarrow M\rightarrow E(M)\rightarrow \frac{E(M)}{M}\rightarrow 0$ is special $\X$-pure. Therefore,  ${\rm Ext}_{R}^{1}(K_{n-1},M)=0$, and so from ${\rm Ext}_{R}^{1}(K_{n-1}, M)\cong{\rm Ext}_{R}^{n}(U, M)$ we get that $M$ is $n$-$\X$-injective.
}
\end{proof}

The following lemma is a generalization of \cite[Exercise 40]{Stenst2}.
\begin{lem}\label{lm} 
Let  $\X$ be a class of  $R$-modules and $0\rightarrow A\rightarrow B\rightarrow C\rightarrow 0$ be a short exact sequence of $R$-modules . Then, the following statements are equivalent:
\begin{enumerate}
\item [\rm (1)]
The exact sequence $0\rightarrow A\rightarrow B\rightarrow C\rightarrow 0$     is special $\X$-pure;
\item [\rm (2)]
The sequence $0\rightarrow {\rm Hom}_{R}(K_{n-1},A)\rightarrow {\rm Hom}_{R}(K_{n-1},B)\rightarrow {\rm Hom}_{R}(K_{n-1},C)\rightarrow 0$ 
is exact for every special $\X$-presented $K_{n-1}$;
\item [\rm (3)]
The short exact  sequence  $0\rightarrow C^{*}\rightarrow B^{*}\rightarrow A^{*}\rightarrow 0$ is a split special exact sequence.
\end{enumerate}
\end{lem}

\begin{prop}\label{3.8iu} 
Let  $\X$ be a class of  $R$-modules. Then:
\begin{enumerate}
\item [\rm (1)]
Every special $\X$-pure submodule of an $n$-$\X$-flat right $R$-module is $n$-$\X$-flat.
\item [\rm (2)]
Every special $\X$-pure submodule of an  $R$-module is $n$-$\X$-injective.
\end{enumerate}
\end{prop}
\begin{proof}
{(1) 
Let $A$ be a special $\X$-pure submodule of an $n$-$\X$-flat  right $R$-module $B$. Then, by
Lemma \ref{lm}, the sequence $0\rightarrow (\frac{B}{A})^{*}\rightarrow B^{*}\rightarrow A^{*}\rightarrow 0$ is a split special exact sequence. By \cite[Lemma 2.8]{D.Benn}, $B^{*}$ is $n$-$\X$-injective. Then, from \cite[lemma 2.7]{D.Benn} and Lemma \ref{lm},  we deduce that $A$ is $n$-$\X$-flat.

(2) Let $A$ be a special $\X$-pure submodule of an $R$-module $B$. Then, the exact sequence $0\rightarrow A\rightarrow B\rightarrow \frac{B}{A}\rightarrow 0$  is special $\X$-pure. So, by Proposition \ref{3.8nn}, $A$ is 
$n$-$\X$-injective.
}
\end{proof}
\begin{rem}\label{21}
(1)
Every flat right $R$-module is $n$-$\X$-flat.

(2)
Every injective left (resp.,  right) $R$-module is $n$-$\X$-injective.

(3)
If $U\in\X_m$, then $U\in\X_n$ for any $m\geq n.$ 
\end{rem}

A ring $R$ is called self left $n$-$\X$-injective if $R$ is an $n$-$\X$-injective left $R$-module.\\
Let $\F$ be a class of $R$-modules and $M$ an $R$-module. Following \cite{Cri-Tor-Monic},  we say
that a morphism $f : F\rightarrow M$ is called an
$\F$-precover of $M$ if $F\in\F$ and ${\rm Hom}_{R}(F^{'}, F) \rightarrow {\rm Hom}_{R}(F^{'},M)\rightarrow 0$ is exact
for all $F^{'}\in\F$.  An
$\F$-precover $f\in {\rm Hom}_{R}(F,M)$ of $M$ is called an $\F$-cover if every endomorphism $g : F\rightarrow F$ with 
$fg = f$ is an automorphism. Dually,
the notions of $\F$-preenvelopes and $\F$-envelopes are defined.

\begin{thm}\label{3.8bv} 
Let $R$ be a left $n$-$\X$-coherent ring and $\X$ be a class of  $R$-modules.  Then, the following statements are equivalent:
\begin{enumerate}
\item [\rm (1)]
$R$ is  self left $n$-$\X$-injective;
\item [\rm (2)]
For any  $R$-module, there is an epimorphism $\mathscr{XI}$-cover;
\item [\rm (3)]
 For any right $R$-module, there is a monomorphic $\mathscr{XF}$-preenvelope;
\item [\rm (4)]
Every injective right $R$-module is $n$-$\X$-flat;
\item [\rm (5)]
Every $1$-$\X$-injective right $R$-module is $n$-$\X$-flat;
\item [\rm (6)]
Every $n$-$\X$-injective right  $R$-module is $n$-$\X$-flat;
 \item [\rm (7)] Every flat  $R$-module is $n$-$\X$-injective.
\end{enumerate}
\end{thm}
\begin{proof}
{$(1)\Longrightarrow (3)$ By \cite[Theorem 2.16]{D.Benn}, every right $R$-module $N$ has an $n$-$\X$-flat preenvelope $f: N\rightarrow F$. By \cite[Theorem 2.13]{D.Benn}, $R^*$ is $n$-$\X$-flat, and so $\prod R^{*}$ is $n$-$\X$-flat by \cite[Theorem 2.6]{D.Benn}. On the other hand, $R^*$ is a cogenerator, so  a monomorphism of the form  $g: N\rightarrow \prod R^{*}$ exists. Hence, there exists a  homomorphism $ h: F\rightarrow \prod R^{*}$ such that $hf=g$ which implies that $f$ is monic.

$(3)\Longrightarrow (4)$ Let $E$ be an injective right $R$-module. By (3), there  is $f: E\rightarrow F$       a  monic $n$-$\X$-flat preenvelope  of $E$. So, the    sequence $0\rightarrow E\rightarrow F\rightarrow\frac{F}{E}\rightarrow 0$ splits, hence   $E$ is $n$-$\X$-flat.

$(3)\Longrightarrow (5)$ The proof  is similar to the one of $(3)\Longrightarrow (4)$.

$(4)\Longrightarrow (6)$ Let $N$ be an $n$-$\X$-injective right $R$-module. Then, by Proposition \ref{3.8nn},  the exact sequence $0\rightarrow N\rightarrow E(N)\rightarrow\frac{E(N)}{N}\rightarrow 0$ is special $\X$-pure. 
Since by (3) $E(N)$ is $n$-$\X$-flat, then from Proposition \ref{3.8iu}, we deduce that $N$ is $n$-$\X$-flat.

$(5)\Longrightarrow (4)$ is clear by Remark \ref{21}.

$(4)\Longrightarrow (1)$ By (4), $R^*$ is $n$-$\X$-flat since $R^*$ is injective. So, $R$ is  self left $n$-$\X$-injective by \cite[Theorem 2.13]{D.Benn}.

 $(6\Rightarrow 7)$  Let $F$ be a flat  $R$-module, then $F^*$ is injective, so $F^*$ is $n$-$\X$-flat by (6),
   and hence $F$ is $n$-$\X$-injective.\\
   $(7\Rightarrow 2)$ For any  $R$-module $M$, there is an $\mathscr{XI}_n$-cover $f:C\rightarrow M$.
    Notice that $R$ is an $n$-$\X$-injective $R$-module, so $f$ is an epimorphism.\\
   $( 2 \Rightarrow 1)$ By hypothesis, $R$ has an epimorphism $\mathscr{XI}_n$-cover $f:D\rightarrow R$,
   then we have a split exact sequence $0\rightarrow Kerf \rightarrow D\rightarrow R\rightarrow 0$ with $D$ is $n$-$\X$-injective. Then, $R$ is $n$-$\X$-injective as a left $R$-module.
}
\end{proof}

\begin{prop}\label{3.8bb} 
Let $R$ be a left $n$-$\X$-coherent ring and $\X$ be a class of   $R$-modules.  If $\{A_i\}_{i\in I}$ is a family of $R$-modules, then $\bigoplus_{i\in I} A_i$ is $n$-$\X$-injective if and only if every $A_i$ is $n$-$\X$-injective.
\end{prop}
\begin{proof}
{Assume that $U\in\X_n$. So, there exists a special exact sequence $0\rightarrow K_n\rightarrow F_{n-1}\rightarrow K_{n-1}\rightarrow 0$ of $\X_n$. Since $R$ is $n$-$\X$-coherent,  we conclude that $U\in\X_{n+1}$ and $K_n$ is special $\X$-presented. So, if $\{A_i\}_{i\in I}$ is a family of $n$-$\X$-injective  $R$-modules, we have that
$${\rm Hom}(K_n, \bigoplus_{i\in I} A_i)\cong \bigoplus_{i\in I}{\rm Hom}(K_n, A_i).$$ One easily gets that
$${\rm Ext}_{R}^{n}(U, \bigoplus_{i\in I} A_i)\cong{\rm Ext}_{R}^{1}(K_n, \bigoplus_{i\in I} A_i)\cong\bigoplus_{i\in I}{\rm Ext}_{R}^{1}(K_n,  A_i)\cong\bigoplus_{i\in I}{\rm Ext}_{R}^{n}(U,  A_i).  $$
}
\end{proof}


\section{Gorenstein $n$-$\X$-injective  and  $n$-$\X$-flat modules}
In this section, we investigate  Gorenstein $n$-$\X$-injective and Gorenstein $n$-$\X$-flat  modules which are defined below. Then, by using of results of Section 3, some characterizations of them are given.

\begin{Def}\label{2.76}
Let $R$ be a ring and $\X$ be a class of  $R$-modules. Then:
\begin{enumerate}
\item [\rm (1)]
 An $R$-module  $G$  is  called  Gorenstein $n$-$\X$-injective, if there exists   an  exact sequence of  $n$-$\X$-injective  $R$-modules:
$${\mathbf{A}}= \cdots\longrightarrow A_1\longrightarrow A_{0}\longrightarrow A^0\longrightarrow
A^1\longrightarrow\cdots$$ with $G=\ker(A^0\rightarrow A^1)$ such that ${\rm Hom}_{R}(K_{n-1},{\mathbf{A}})$ is an exact sequence  whenever $K_{n-1}$  is special $\X$-presented with ${\rm pd}_R(K_{n-1})<\infty.$
\item [\rm (1)]
 An $R$-module  $G$  is called  Gorenstein $n$-$\X$-flat right $R$-module if there exists an  exact sequence of $n$-$\X$-flat right  $R$-modules:
$${\mathbf{F}}= \cdots\longrightarrow F_1\longrightarrow F_{0}\longrightarrow F^0\longrightarrow
F^1\longrightarrow\cdots$$ with $G=\ker(F^0\rightarrow F^1)$ such that ${\mathbf{F}}\otimes_{R}K_{n-1}$  is an exact   sequence   whenever $K_{n-1}$  is special $\X$-presented with ${\rm fd}_R(K_{n-1})<\infty.$

\end{enumerate}
\end{Def}
For example, 
if $\X$ is a class of all cyclic $R$-modules, then every Gorenstein $1$-$\X$-injective  $R$-module is Gorenstein FP-injective, and every Gorenstein $1$-$\X$-flat right $R$-module is Gorenstein flat, see \cite{D.Benn, Z.G}.
 
 \begin{rem}\label{2} (1)
Every $n$-$\X$-flat right $R$-module is  Gorenstein $n$-$\X$-flat.

(2)
Every $n$-$\X$-injective  $R$-module is  Gorenstein $n$-$\X$-injective.

(3)
In Definition \ref{2.76}, one easily gets that
each $\ker(A_i\rightarrow A_{i-1})$, $\ker(A^i\rightarrow A^{i+1})$ and $\ker(F_i\rightarrow F_{i-1})$, $K^i=\ker(F^i\rightarrow F^{i+1})$ are Gorenstein $n$-$\X$-injective and Gorenstein $n$-$\X$-flat, respectively.
\end{rem}

\begin{lem}\label{2.09}
Let $R$ be a left $n$-$\X$-coherent ring and $\X$ be a class of   $R$-modules.  If $K_{n-1}$ is a special $\X$-presented    $R$-module with ${\rm fd}_R(K_{n-1})<\infty$, then ${\rm pd}_{R}(K_{n-1})<\infty$.
\end{lem}
\begin{proof}
{ If ${\rm fd}_{R}(K_{n-1})=m<\infty$, then there exists $U\in\X_n$ such that ${\rm fd}_{R}(U)\leq n+m$. We 
show that ${\rm pd}_{R}(U)\leq n+m$. Since $R$ is $n$-$\X$-coherent, the projective resolution 
$\cdots \rightarrow F_{n+1}\rightarrow F_{n}\rightarrow\cdots \rightarrow F_{0}\rightarrow U\rightarrow 0,$ where any $F_i$ is finitely generated free, exists. On the other hand, the above exact sequence  is a flat resolution.  So by \cite[Proposition 8.17]{Rot2},  $(n+m-1)$-syzygy is flat. Hence, the exact sequence
$0\rightarrow K_{n+m-1}\rightarrow F_{n+m-1}\rightarrow\cdots \rightarrow F_{0}\rightarrow U\rightarrow 0$ is a flat resolution. Now, a simple observation shows that if $n\geq m$ or $n<m$, $K_{n+m-1}$ is finitely presented and consequently by \cite[Theorem 3.56]{Rot2}, $K_{n+m-1}$ is projective and therefore, ${\rm pd}_{R}(U)\leq n+ m$ if and only if ${\rm pd}_{R}(K_{n-1})\leq m$. }
\end{proof} 

In the following theorem, we show that in the case of left $n$-$\X$-coherent rings,  Gorenstein $n$-$\X$-flat and Gorenstein $n$-$\X$-injective are determined via only the existence  of the corresponding exact  complexes.

\begin{thm}\label{2.3}
Let  $R$ be a left $n$-$\X$-coherent ring and  $\X$ be a class of   $R$-modules. Then:
\begin{enumerate}
\item [\rm (1)]
 A  right $R$-module $G$ is Gorenstein $n$-$\X$-flat if and only if there is an exact sequence 
$${\mathbf{F}}= \cdots\longrightarrow F_1\longrightarrow F_{0}\longrightarrow F^0\longrightarrow
F^1\longrightarrow\cdots$$ of  $n$-$\X$-flat right $R$-modules  such that $G=\ker(F^0\rightarrow F^1)$.
\item [\rm (2)]
An $R$-module $G$ is Gorenstein $n$-$\X$-injective  if and only if there is an exact sequence
$${\mathbf{A}}= \cdots\longrightarrow A_1\longrightarrow A_{0}\longrightarrow A^0\longrightarrow
A^1\longrightarrow\cdots$$
of $n$-$\X$-injective   $R$-modules such that $G=\ker(A^0\rightarrow A^1)$.
\end{enumerate}
\end{thm}
\begin{proof}
{(1) ($\Longrightarrow$) follows by  definition.

($\Longleftarrow$)   By definition, it suffices to show that $\mathbf{F}\otimes_{R}K_{n-1}$ is exact for
every  special $\X$-presented $K_{n-1}$ with ${\rm fd}_R(K_{n-1})<\infty$. By Lemma \ref{2.09}, ${\rm pd}_{R} (K_{n-1})< \infty$.  Let ${\rm pd}_{R} (K_{n-1}) = m $. We prove  by  induction on $m$. The case $m =0$ is clear. Assume that $m \geq 1$. There exists a   special exact sequence
$0\rightarrow K_n\rightarrow P_{n-1}\rightarrow K_{n-1} \rightarrow 0$ of $U\in\X_n$, where $P_{n}$ is projective finitely generated. Now, from the $n$-$\X$-coherence of $R$, we deduce that $K_n$ is special $\X$-presented. Also, ${\rm pd}_{R}(K_n) \leq m-1$. So, the following short exact sequence of complexes exists:
\begin{center}
$
\begin{array}{ccccccccc}
& \vdots&\vdots &\vdots&\\
& \downarrow & \downarrow &\downarrow & \\
0 \longrightarrow &F_1\otimes_{R}K_n & \longrightarrow F_{1}\otimes_{R}P_{n-1}&\longrightarrow F_1\otimes_{R}K_{n-1}\longrightarrow 0 \\
& \downarrow & \downarrow &\downarrow & \\
0 \longrightarrow &F_0\otimes_{R}K_n & \longrightarrow F_0\otimes_{R}P_{n-1}&\longrightarrow F_0\otimes_{R}K_{n-1}\longrightarrow 0 \\
& \downarrow &\downarrow &\downarrow & \\
0 \longrightarrow &F^0\otimes_{R}K_n & \longrightarrow F^0\otimes_{R}P_{n-1}&\longrightarrow F^0\otimes_{R}K_{n-1}\longrightarrow 0 \\
& \downarrow &\downarrow &\downarrow & \\
0 \longrightarrow &F^1\otimes_{R}K_n & \longrightarrow F^1\otimes_{R}P_{n-1}&\longrightarrow F^1\otimes_{R}K_{n-1}\longrightarrow 0 \\
& \downarrow &\downarrow &\downarrow & \\
& \vdots&\vdots &\vdots&\\
& \parallel & \parallel &\parallel& \\
0 \longrightarrow &{\mathbf{F}}\otimes_{R}K_n& \longrightarrow {\mathbf{F}}\otimes_{R}P_{n-1}&\longrightarrow {\mathbf{F}}\otimes_{R}K_{n-1}\longrightarrow 0. \\

\end{array}
$
\end{center}
\noindent By induction, ${\mathbf{F}}\otimes_{R}P_{n-1}$ and ${\mathbf{F}}\otimes_{R}K_n$ are exact, hence ${\mathbf{F}}\otimes_{R}K_{n-1}$ is exact by \cite[Theorem 6.10]{Rot2}. 

(2) ($\Longrightarrow$)  This is a direct consequence of the definition.

($\Longleftarrow$) Let $K_{n-1}$ be  a special $\X$-presented $R$-module with ${\rm pd}_{R} (K_{n-1})< \infty$. Then, similar proof to that of (1) shows that ${\rm Hom}_{R}(K_{n-1},{\mathbf{A}})$ is exact and hence $G$ is Gorenstein $n$-$\X$-injective.
}
\end{proof}

\begin{cor}\label{2.h}
Let $R$ be a left $n$-$\X$-coherent ring and $\X$ be a class of   $R$-modules. Then, for any   $R$-module $G$, the following assertions are equivalent:
\begin{enumerate}
\item [\rm (1)]
$G$ is Gorenstein $n$-$\X$-injective;
\item [\rm (2)]
There is an exact sequence $\cdots\rightarrow A_1\rightarrow A_{0}\rightarrow G\rightarrow 0$ of   $R$-modules, where every $A_i$ is $n$-$\X$-injective;
\item [\rm (3)]
There is a short exact sequence $0\rightarrow L\rightarrow M\rightarrow G\rightarrow 0$ of   $R$-modules, where
$M$ is $n$-$\X$-injective and $L $ is Gorenstein $n$-$\X$-injective.
\end{enumerate}
\end{cor}
\begin{proof}
{
$(1)\Longrightarrow (2)$ and $(1)\Longrightarrow (3)$ follow from the definition.

$(2)\Longrightarrow (1)$There is an exact sequence
$$0\longrightarrow G\longrightarrow I^{0}\longrightarrow I^1\longrightarrow \cdots$$
where every $I^{i}$ is injective for any $i\geq 0$. By Remark \ref{21}, each $I^{i}$ is $n$-$\X$-injective.  So, an exact sequence
$$ \cdots\longrightarrow A_1\longrightarrow A_{0}\longrightarrow I^0\longrightarrow
I^1\longrightarrow\cdots$$
of $n$-$\X$-injective   modules  exists, where $G={\rm ker}(I^0\rightarrow I^1)$. Therefore, $G$ is Gorenstein $n$-$\X$-injective by Theorem \ref{2.3}.

$(3)\Longrightarrow (2)$ Assume there is an exact sequence
$$ 0\longrightarrow L\longrightarrow M \longrightarrow G\longrightarrow 0, \ \ (1)$$  where $M$ is $n$-$\X$-injective and $L $ is Gorenstein $n$-$\X$-injective. Since $L$ is Gorenstein $n$-$\X$-injective, there is an exact sequence
$$\cdots\longrightarrow A_2^{'}\longrightarrow A_1^{'}\longrightarrow A_0^{'}\longrightarrow L\longrightarrow 0,\ \ (2)$$
where every $A_i^{'}$ is $n$-$\X$-injective.
Assembling the sequences $(1)$ and $(2)$, we
get the exact sequence 
$$\cdots\longrightarrow A_2^{'}\longrightarrow A_1^{'}\longrightarrow A_0^{'}\longrightarrow M\longrightarrow G\longrightarrow 0,$$ where $M$ and $A_i^{'}$ are $n$-$\X$-injective,
 as desired.}
\end{proof}

\begin{cor}\label{2.ty}
Let $R$ be a left $n$-$\X$-coherent ring and $\X$ be a class of   $R$-modules. Then, for any right $R$-module $G$,
the following assertions are equivalent:
\begin{enumerate}
\item [\rm (1)]
$G$ is  Gorenstein $n$-$\X$-flat;
\item [\rm (2)]
There is an exact sequence $0\rightarrow G\rightarrow B^{0}\rightarrow B^{1}\rightarrow \cdots$ of right $R$-modules, where every $B^i$ is $n$-$\X$-flat;
\item [\rm (3)]
There is a short exact sequence $0\rightarrow G\rightarrow M\rightarrow L\rightarrow 0$ of right $R$-modules, where
$M$ is $n$-$\X$-flat and $L $ is Gorenstein $n$-$\X$-flat.
\end{enumerate}
\end{cor}
\begin{proof}
{
$(1)\Longrightarrow (2)$ and $(1)\Longrightarrow (3)$ follow from the definition.

$(2)\Longrightarrow (1)$ For any right $R$-module $G$, there is an exact sequence
$$\cdots\longrightarrow P_{1}\longrightarrow P_{0}\longrightarrow G\longrightarrow 0, $$
where any $P_{i}$ is flat  for any $i\geq 0$.  By Remark \ref{21}, every $P_i$ is $n$-$\X$-flat. Thus, there is an exact sequence
$$ \cdots\longrightarrow P_1\longrightarrow P_{0}\longrightarrow B^0\longrightarrow
B^1\longrightarrow\cdots$$
of  $n$-$\X$-flat right modules, where $G={\rm ker}(B^0\rightarrow B^1)$. Therefore, by Theorem \ref{2.3}, $G$ is Gorenstein $n$-$\X$-flat. 

$(3)\Longrightarrow (2)$ Assume there is an    exact sequence
$$ 0\longrightarrow G\longrightarrow M \longrightarrow L\longrightarrow 0, \ \ (1)$$   where
$M$  is $n$-$\X$-flat and $L $ is Gorenstein $n$-$\X$-flat. Since $L$ is Gorenstein $n$-$\X$-flat, there is an exact sequence
$$0\longrightarrow L\longrightarrow F^0\longrightarrow F^1\longrightarrow F^2\longrightarrow \cdots,\ \ (2)$$
where every $F^i$ is $n$-$\X$-flat.
Assembling the sequences $(1)$ and $(2)$, we
get the exact sequence
$$0\longrightarrow G\longrightarrow M\longrightarrow F^0\longrightarrow F^1\longrightarrow F^2\longrightarrow \cdots,$$ where $M$ and any $ F^i$ are $n$-$\X$-flat, as desired.
}
\end{proof}

\begin{prop}\label{2.jj}
Let $\X$ be a class of   $R$-modules. Then: 
\begin{enumerate}
\item [\rm (1)]
Every direct product of Gorenstein $n$-$\X$-injective $R$-modules is a Gorenstein $n$-$\X$-injective $R$-module.
\item [\rm (2)]
Every direct sum of Gorenstein $n$-$\X$-flat right $R$-modules is a Gorenstein $n$-$\X$-flat $R$-module.
\end{enumerate}
\end{prop}
\begin{proof}
{(1) Let $U\in\X_n$ and    $\{A_i\}_{i\in I}$ be a family of  $n$-$\X$-injective   $R$-modules. Then, by \cite[Lemma 2.7]{D.Benn}, $\prod A_i$ is $n$-$\X$-injective. So, if   $\{G_i\}_{i\in I}$ is a family of Gorenstein $n$-$\X$-injective   $R$-modules, then
 the following corresponding  exact sequences  of $n$-$\X$-injective  $R$-modules
 $${\mathbf{A_i}}=\cdots \longrightarrow  (A_{i})_{1}\longrightarrow (A_{i})_{0}\longrightarrow (A_{i})^{0}\longrightarrow (A_{i})^{1}\longrightarrow\cdots,$$
  where $G_i={\rm ker}( (A_{i})^{0}\rightarrow  (A_{i})^{1})$, induce the following exact sequence of $n$-$\X$-injective $R$-modules:
 $$\prod_{i\in I}{\mathbf{A_i}}=\cdots \longrightarrow  \prod_{i\in I}(A_{i})_{1}\longrightarrow \prod_{i\in I}(A_{i})_{0}\longrightarrow \prod_{i\in I}(A_{i})^{0}\longrightarrow \prod_{i\in I}(A_{i})^{1}\longrightarrow\cdots,$$
  where $\prod_{i\in I} G_i={\rm ker}( \prod _{i\in I}(A_{i})^{0}\rightarrow \prod_{i\in I} (A_{i})^{1})$. If $K_{n-1}$ is special $\X$-presented, then $${\rm Hom}_{R}(K_{n-1},\prod_{i\in I}{\mathbf{A_i}})\cong\prod_{i\in I}{\rm Hom}_{R}(K_{n-1},{\mathbf{A_i}}).$$
  By hypothesis, ${\rm Hom}_{R}(K_{n-1},{\mathbf{A_i}})$ is exact, and consequently $\prod_{i\in I}G_i$ is Gorenstein $n$-$\X$-injective.

(2)
Let $U\in\X_n$ and    $\{I_i\}_{i\in J}$ be a family of  $n$-$\X$-flat right $R$-modules. Then, by \cite[Lemma 2.7]{D.Benn}, $\bigoplus_{i\in J} I_i$ is $n$-$\X$-flat. So, if   $\{G_i\}_{i\in J}$ is a family of Gorenstein $n$-$\X$-flat right $R$-modules, then
 the following corresponding   exact sequences  of  $n$-$\X$-flat  right $R$-modules 
 $${\mathbf{I_i}}=\cdots \longrightarrow  (I_{i})_{1}\longrightarrow (I_{i})_{0}\longrightarrow (I_{i})^{0}\longrightarrow (I_{i})^{1}\longrightarrow\cdots,$$
  where $G_i={\rm ker}( (I_{i})^{0}\rightarrow  (I_{i})^{1})$, induces the following exact sequence of $n$-$\X$-flat right $R$-modules:
 $$\bigoplus_{i\in J}{\mathbf{I_i}}=\cdots \longrightarrow  \bigoplus_{i\in J}(I_{i})_{1}\longrightarrow \bigoplus_{i\in J}(I_{i})_{0}\longrightarrow \bigoplus_{i\in J}(I_{i})^{0}\longrightarrow \bigoplus_{i\in J}(I_{i})^{1}\longrightarrow\cdots,$$
  where $\bigoplus_{i\in J}G_i={\rm ker}( (\bigoplus_{i\in J}I_{i})^{0}\rightarrow  (\bigoplus_{i\in J}I_{i})^{1})$. If $K_{n-1}$ is special $\X$-presented, then $$(\bigoplus_{i\in J}{\mathbf{I_i}}\otimes_{R}K_{n-1})\cong\bigoplus_{i\in J} ({\mathbf{I_i}}\otimes_{R}K_{n-1}).$$
  By hypothesis, $({\mathbf{I_i}}\otimes_{R}K_{n-1})$ is exact, and consequently $\bigoplus_{i\in J}G_i$ is Gorenstein $n$-$\X$-flat.}
\end{proof}
Now, we study the Gorenstein $n$-$\X$-injectivity and Gorenstein $n$-$\X$-flatness of modules in short exact sequences.

\begin{prop}\label{2.7}
Let $R$ be a left $n$-$\X$-coherent ring and $\X$ be a class of   $R$-modules.  Then:
\begin{enumerate}
\item [\rm (1)]
Let $0\rightarrow A\rightarrow G\rightarrow N\rightarrow 0$ be an exact sequence of   $R$-modules. If $A$ and $N$ are Gorenstein $n$-$\X$-injective, then $G$ is Gorenstein $n$-$\X$-injective.
\item [\rm (2)]
Let $0\rightarrow K\rightarrow G\rightarrow B\rightarrow 0$ be an exact sequence of right $R$-modules. If $K$ and $B$ are Gorenstein $n$-$\X$-flat, then $G$ is Gorenstein $n$-$\X$-flat.
\end{enumerate}
\end{prop}
\begin{proof}
{
(1)  Since $A$ and $N$ are Gorenstein $n$-$\X$-injective, by Corollary $\ref{2.h}$, there exist exact sequences 
$0\rightarrow K\rightarrow A_{0}\rightarrow A\rightarrow 0$ and $0\rightarrow L\rightarrow N_{0}\rightarrow N\rightarrow 0$ of   $R$-modules, where $A_{0}$ and $N_0$ are $n$-$\X$-injective and also, $K$ and $L$ are Gorenstein $n$-$\X$-injective. 
Now, we consider the following commutative diagram:
\begin{center}
$
\begin{array}{ccccccccc}
 & & 0 & & 0 & &0 & & \\
& & \downarrow & &\downarrow & & \downarrow & & \\
0 &\longrightarrow & K & \longrightarrow & K\oplus L&\longrightarrow &L & \longrightarrow & 0 \\
 & & \downarrow & &\downarrow & & \downarrow & & \\
0 &\longrightarrow & A_0 & \longrightarrow & A_{0}\oplus N_{0} &\longrightarrow &N_0 & \longrightarrow & 0 \\
& & \downarrow & &\downarrow & & \downarrow & & \\
0 & \longrightarrow & A& \longrightarrow & G& \longrightarrow & N & \longrightarrow & 0 \\
& & \downarrow & &\downarrow & & \downarrow & & \\
 & & 0 & & 0 & &0 & & \\
\end{array}
$
\end{center}
The exactness of the middle horizontal sequence where $A_0$ and $N_0$ are $n$-$\X$-injective, implies that $A_{0}\oplus N_{0}$ is $n$-$\X$-injective by \cite[Lemma 2.7]{D.Benn}. Also, $K\oplus L$ is Gorenstein $n$-$\X$-injective by Proposition \ref{2.jj}(1). Hence, from the middle vertical sequence
and Corollary $\ref{2.h}$, we deduce that $G$ is Gorenstein $n$-$\X$-injective.

(2)
  Since $K$ and $B$ are Gorenstein $n$-$\X$-flat, by Corollary $\ref{2.ty}$, there exist exact sequences 
$0\rightarrow K\rightarrow K_{0}\rightarrow L_1\rightarrow 0$ and $0\rightarrow B\rightarrow B_{0}\rightarrow L_{1}^{'}\rightarrow 0$ of $R$-modules, where $K_{0}$ and $B_0$ are $n$-$\X$-flat and also, $L_1$ and $L_{1}^{'}$ are Gorenstein $n$-$\X$-flat. 
Now, we consider the following commutative diagram:
\begin{center}
$
\begin{array}{ccccccccc}
 & & 0 & & 0 & &0 & & \\
& & \downarrow & &\downarrow & & \downarrow & & \\
0 &\longrightarrow & K & \longrightarrow & G&\longrightarrow &B & \longrightarrow & 0 \\
 & & \downarrow & &\downarrow & & \downarrow & & \\
0 &\longrightarrow & K_0 & \longrightarrow & K_{0}\oplus B_{0} &\longrightarrow &B_0 & \longrightarrow & 0 \\
& & \downarrow & &\downarrow & & \downarrow & & \\
0 & \longrightarrow & L_1& \longrightarrow & L_{1}\oplus L_{1}^{'} & \longrightarrow & L_{1}^{'} & \longrightarrow & 0 \\
& & \downarrow & &\downarrow & & \downarrow & & \\
 & & 0 & & 0 & &0 & & \\
\end{array}
$
\end{center}
The exactness of the middle horizontal sequence with $K_0$ and $B_0$ are $n$-$\X$-flat, implies that $K_{0}\oplus B_{0}$ is $n$-$\X$-flat by \cite[Lemma 2.7]{D.Benn}. Also, $L_{1}\oplus L_{1}^{'}$ is Gorenstein $n$-$\X$-flat by Proposition \ref{2.jj}(2). Hence from the middle vertical sequence
and Corollary $\ref{2.ty}$, we deduce that $G$ is Gorenstein $n$-$\X$-flat.
}
\end{proof}
The left $n$-$\X$-injective dimension of an $R$-module $M$, denoted by $\id_{\X_n}(M)$,  is defined to be
 the least non-negative integer $m$ such that $\Ext_R^{n+m+1}(U,M)=0$  for any $U\in \X_n$.  The left $n$-$\X$-flat dimension of a right $R$-module $M$, denoted by $\fd_{\X_n}(M)$, is defined to be
 the least non-negative integer $m$ such that $\Tor_{n+m+1}^R(M,U)=0$  for any $U\in \X_n$. If $G$ is Gorenstein $n$-$\X$-injective, then $\id_{\X_n}(G)=m$ if there is an exact sequence
 $$0\longrightarrow A_{m}\longrightarrow\cdots\longrightarrow A_{1} \longrightarrow A_0
\longrightarrow G \longrightarrow 0$$
or
  an exact sequence
 $$0\longrightarrow G\longrightarrow A^{0} \longrightarrow A^1
\longrightarrow\cdots \longrightarrow A^{m} \longrightarrow 0$$ of $n$-$\X$-injective   $R$-modules.
Similarly, if $G$ is Gorenstein $n$-$\X$-flat and $\fd_{\X_n}(G)=m$, then the above exact sequences for $n$-$\X$-flat right $R$-modules exists.

The following theorems are   generalizations of Corollaries \ref{2.h} and \ref{2.ty} and Proposition \ref{2.7}.

\begin{thm}\label{2.po} 
Let  $R$ be a left $n$-$\X$-coherent ring and $\X$ be a class of   $R$-modules which is closed under kernels of epimorphisms. Then, for every   $R$-module $G$, the following statements are equivalent:
\begin{enumerate}
\item [\rm (1)]
$G$ is Gorenstein $n$-$\X$-injective;
\item [\rm (2)]
There exists an $n$-$\X$-injective resolution of $G$:
$$
\cdots\stackrel{\displaystyle f_{3}} \longrightarrow A_{2}\stackrel{\displaystyle f_{2}}\longrightarrow A_{1}\stackrel{\displaystyle f_{1}} \longrightarrow A_0
\stackrel{\displaystyle f_{0}}\longrightarrow G\longrightarrow 0
$$
 such that $\bigoplus_{i=0}^{\infty}{\rm Im}(f_i)$ is Gorenstein $n$-$\X$-injective;
\item [\rm (3)]
There exists an exact sequence 
$$
\cdots\stackrel{\displaystyle f_{3}} \longrightarrow A_{2}\stackrel{\displaystyle f_{2}}\longrightarrow A_{1}\stackrel{\displaystyle f_{1}} \longrightarrow A_0
\stackrel{\displaystyle f_{0}}\longrightarrow G\longrightarrow 0
$$
 of   $R$-modules, where $A_i$ has finite $n$-$\X$-injective dimension for any $i\geq 0$, such that $\bigoplus_{i=0}^{\infty}{\rm Im}(f_i)$ is Gorenstein $n$-$\X$-injective.
\end{enumerate}
\end{thm}
\begin{proof} 
 $(1)\Longrightarrow (2)$ By Corollary \ref{2.h},  there is an exact sequence
$$
\cdots\stackrel{\displaystyle f_{3}} \longrightarrow A_{2}\stackrel{\displaystyle f_{2}}\longrightarrow A_{1}\stackrel{\displaystyle f_{1}} \longrightarrow A_0
\stackrel{\displaystyle f_{0}}\longrightarrow G\longrightarrow 0,
$$
where every $A_i$ is $n$-$\X$-injective.
Consider the following exact sequences:
$$\cdots \longrightarrow A_{2} \longrightarrow A_{1}  \longrightarrow   A_0
 \longrightarrow  {\rm Im}(f_0) \longrightarrow 0,$$
$$\cdots \longrightarrow A_{3} \longrightarrow A_{2}  \longrightarrow   A_1
 \longrightarrow  {\rm Im}(f_1) \longrightarrow 0,$$
$$    \vdots \ \ \ \ \  \ \ \  \ \ \vdots \ \ \ \ \ \ \ \ \ \ \vdots \ \ \  \ \ \ \ \ \ \ \ \ \ \ \ \vdots$$
 By Proposition \ref{3.8bb}, $\bigoplus_{i\in I}A_i $ is $n$-$\X$-injective. Thus, there exists an exact sequence
$$\cdots \longrightarrow \bigoplus_{i\geq 2}A_i \longrightarrow \bigoplus_{i\geq 1}A_i \longrightarrow   \bigoplus_{i\geq 0}A_i
 \longrightarrow  \bigoplus_{i=0}^{\infty}{\rm Im}(f_i) \longrightarrow 0$$ 
 of $n$-$\X$-injective   $R$-modules. Consequently, Corollary \ref{2.h} implies that $ \bigoplus_{i=0}^{\infty}{\rm Im}(f_i)$
 is Gorenstein $n$-$\X$-injective. 
 
 $(2)\Longrightarrow (3)$ trivial.
 
$(3)\Longrightarrow (1)$ Let 
$$
\cdots\stackrel{\displaystyle f_{3}} \longrightarrow A_{2}\stackrel{\displaystyle f_{2}}\longrightarrow A_{1}\stackrel{\displaystyle f_{1}} \longrightarrow A_0
\stackrel{\displaystyle f_{0}}\longrightarrow G\longrightarrow 0
$$
be  an exact sequence  of   $R$-modules, where $A_i$ has finite $n$-$\X$-injective dimension. By Corollary \ref{2.h}, it is sufficient to prove that $A_i$ is $n$-$\X$-injective for any 
$i\geq0$. Consider, the short exact sequence $0\rightarrow{\rm Im}(f_{i+1})\rightarrow A_i\rightarrow{\rm Im}(f_{i})\rightarrow 0$ for any $i\geq 0$. Therefore, the short exact sequence $0\rightarrow\bigoplus_{i=0}^{\infty}{\rm Im}(f_{i+1})\rightarrow \bigoplus_{i=0}^{\infty}A_i\rightarrow\bigoplus_{i=0}^{\infty}{\rm Im}(f_{i})\rightarrow 0$ exists. By (3) and Proposition \ref{2.7}(1), $\bigoplus_{i=0}^{\infty}A_i$ is Gorenstein $n$-$\X$-injective. Also, $\bigoplus_{i=0}^{\infty}A_i$ has finite $n$-$\X$-injective dimension. If $\id_{\X_n}(\bigoplus_{i=0}^{\infty}A_i)=k$, then there exists an $n$-$\X$-injective resolution of  $\bigoplus_{i=0}^{\infty}A_i$:
$$0\longrightarrow B_{k} \longrightarrow B_{k-1}  \longrightarrow\cdots \longrightarrow B_0
 \longrightarrow  \bigoplus_{i\in I}A_i\longrightarrow 0.$$
Let $L_{k-1}={\rm ker}(B_{k-1} \rightarrow B_{k-2})$ and $U\in\X_n$. Then, the exact sequence 
$0\rightarrow B_{k}\rightarrow B_{k-1}\rightarrow L_{k-1}\rightarrow0$ 
induces the following exact sequence:
$$0={\rm Ext}_{R}^{n}(U,B_{k-1})\longrightarrow {\rm Ext}_{R}^{n}(U,L_{k-1})\longrightarrow {\rm Ext}_{R}^{n+1}(U,B_{k})\longrightarrow\cdots.$$
By hypothesis, $B_{k}$ is $(n+1)$-$\X$-injective, and also $U\in\X_{n+1}$ since $R$ is $n$-$\X$-coherent. So ${\rm Ext}_{R}^{n+1}(U,B_{k})=0$, and hence ${\rm Ext}_{R}^{n}(U,L_{k-1})=0$. Thus, $L_{k-1}$ is $n$-$\X$-injective. Then, with the same process, we get that  $\bigoplus_{i=0}^{\infty}A_i$ is $n$-$\X$-injective, and so by Proposition \ref{3.8bb}, $A_i$ is $n$-$\X$-injective for any $i\geq 0$.
\end{proof}

For the  following theorem, the proof is similar to that of $(1)\Longrightarrow (2)$, $(2)\Longrightarrow (3)$ and $(3)\Longrightarrow (1)$ in Theorem \ref{2.po}.

\begin{thm}\label{3.100} 
Let $R$ be a left $n$-$\X$-coherent ring and $\X$ be a class of   $R$-modules which  is closed under
kernels of epimorphisms. Then, for every  right $R$-module $G$, the following statements are equivalent:
\begin{enumerate}
\item [\rm (1)]
$G$ is Gorenstein $n$-$\X$-flat;
\item [\rm (2)]
There exists the following right $n$-$\X$-flat resolution of $G$:
$$
0 \longrightarrow G\stackrel{\displaystyle f^{0}}\longrightarrow I^{0} \stackrel{\displaystyle f^{1}}\longrightarrow  I^1\stackrel{\displaystyle f^{2}}
\longrightarrow\cdots
$$
 such that $\bigoplus_{i=0}^{\infty}{\rm Im}(f^i)$ is Gorenstein $n$-$\X$-flat;
\item [\rm (3)]
There exists an exact sequence 
$$
0 \longrightarrow G\stackrel{\displaystyle f^{0}}\longrightarrow I^{0} \stackrel{\displaystyle f^{1}}\longrightarrow  I^1\stackrel{\displaystyle f^{2}}
\longrightarrow\cdots
$$
 of right $R$-modules, where $I_i$ has finite $n$-$\X$-flat dimension for any $i\geq 0$,  such that $\bigoplus_{i=0}^{\infty}{\rm Im}(f^i)$ is Gorenstein $n$-$\X$-flat.
\end{enumerate}
\end{thm}

\section{ $\X$-$FC$-rings }

 A ring $R$ is called left $\X$-$FC$-ring if $R$ is self left $n$-$\X$-injective and left  $n$-$\X$-coherent. In this  section, we investigate properties of Gorenstein  $n$-$\X$-injective and $n$-$\X$-flat modules  over $\X$-$FC$-rings, thus generalizing several classical results. Notice that the notion of $\X$-$FC$-ring generalizes the classical notions of quasi-Frobenius and $FC$ (i.e., IF) rings.

It is well-known that quasi-Frobenius (resp., $FC$) rings can be seen as rings over which all modules are Gorenstein injective (resp., Gorenstein $FP$-injective). Here, we extend this fact as well as other known ones to our new context. 

\begin{prop}\label{3.1} 
Let $\X$ be a class of   $R$-modules. Then, 
every   $R$-module is Gorenstein $n$-$\X$-injective if and only if 
every projective   $R$-module is $n$-$\X$-injective and for any   $R$-module $N$,
${\rm Hom}_R(-, N)$ is exact with respect to all special short exact sequences of $\X_n$ with modules of  finite projective dimension. 
\end{prop}
\begin{proof}
{
$(\Longrightarrow)$ Let $M$ be a  projective
  $R$-module. Then, by hypothesis, $M$ is Gorenstein $n$-$\X$-injective. So, the following  $n$-$\X$-injective resolution of $M$ exists:
$$\cdots \longrightarrow A_1\longrightarrow A_0\longrightarrow M\longrightarrow 0.$$
Since $M$ is projective, $M$ is
$n$-$\X$-injective as a direct summand of $A_0$. Also, by hypothesis and Definition \ref{2.76},  
${\rm Hom}_R(-, N)$ is exact with respect to all special short exact sequences with modules of  finite projective dimension since every  $R$-module $N$ is Gorenstein $n$-$\X$-injective. 

$(\Longleftarrow)$ Choose an  injective resolution of $G$:  $0\rightarrow G\rightarrow E^{0}\rightarrow E^{1}\rightarrow\cdots$ and a projective resolution  of $G$: $\cdots \rightarrow F_{1}\rightarrow F_{0}\rightarrow G\rightarrow 0$, where every $F_i$ is $n$-$\X$-injective by hypothesis. 
 Assembling these resolutions, we get, by Remark \ref{21}, the following exact sequence of $n$-$\X$-injective $R$-modules:
$${\mathbf{A}}=\cdots \longrightarrow F_{1}\longrightarrow F_{0}\longrightarrow E^{0}\longrightarrow E^{1}\longrightarrow\cdots,$$
where $G={\rm ker}(E^0\rightarrow E^1)$, $K^{i}={\rm ker}(E^i\rightarrow E^{i+1})$ and $K_{i}={\rm ker}(F_{i}\rightarrow F_{i-1})$ for any $i\geq 1$.  Let $K_{n-1}$ be a special $\X$-presented module with ${\rm pd}_{R}(K_{n-1}) <\infty$. Then, by hypothesis, we have: $${\rm Ext}_R^{1}(K_{n-1}, G)={\rm Ext}_R^{1}(K_{n-1}, K_i)={\rm Ext}_R^{1}(K_{n-1}, K^{i})=0.$$
So, ${\rm Hom}_R(K_{n-1},{\mathbf{A}})$ is exact, and hence $G$ is Gorenstein $n$-$\X$-injective.}
\end{proof}

\begin{prop}\label{3.vc}
Let $\X$ be a class of  $R$-modules. Then, 
every right $R$-module is Gorenstein $n$-$\X$-flat if and only if 
every injective right $R$-module is $n$-$\X$-flat and for any $R$-module $N$,
$N\otimes_{R}-$ is exact with respect to all special short exact sequences of $\X_n$   with modules of  finite projective dimension. 
\end{prop}
\begin{proof}
{Similar to the proof  of  Proposition \ref{3.1}.} 
\end{proof}

\begin{thm}\label{3.2} 
Let $R$ be a left $n$-$\X$-coherent ring and $\X$ be a class of   $R$-modules. Then,  the following statements are equivalent:
\begin{enumerate}
\item [\rm (1)]
Every  $R$-module  is Gorenstein $n$-$\X$-injective;
\item [\rm (2)]
Every projective   $R$-module is $n$-$\X$-injective;
\item [\rm (3)]
$R$ is self left $n$-$\X$-injective.
\end{enumerate}
\end{thm}
\begin{proof}
{$(1)\Longrightarrow (2)$ and $(2)\Longrightarrow (3)$ hold by Proposition \ref{3.1}. 

$(3)\Longrightarrow (1)$
Let $G$ be an $R$-module and $\cdots\rightarrow F_1\rightarrow F_0\rightarrow G\rightarrow 0$ be any
free resolution of $G$. Then, by  Proposition \ref{3.8bb},  each $F_i$ is $n$-$\X$-injective. Hence, Corollary $\ref{2.h}$ completes the proof.}
\end{proof}

\begin{exs}\label{3.54} 
{Let  $R=k[x^3,x^2,x^{2}y,xy^2,xy,y^2,y^3]$ be  a ring and $\X$ a class of all $1$-presented  $R$-modules.  We claim that $R$ is not $1$-$\X$-injective. Suppose to the contrary, $R$ is $1$-$\X$-injective.  We have $\frac{R}{Rx^2}$ is special $\X$-presented since $Rx^2\cong R$ is special $\X$-generated. Also, ${\rm pd}_{R}(\frac{R}{Rx^2})<\infty$. So by Proposition \ref{3.1} and Theorem \ref{3.2}, $\frac{R}{Rx^2}$ is projective. Therefore, the exact sequence $0\rightarrow Rx^2\rightarrow R\rightarrow \frac{R}{Rx^2}\rightarrow0$ splits. Thus, $Rx^2$ is a direct summand of $R$ and so, $x^2$ is an idempotent, a contradiction.}
\end{exs}

 Let $\X$ be a class of graded    $R$-modules. Then, a graded ring $R$ will be  called $n$-gr-regular if and only if it is $n$-$\X$-regular if and only if every $n$-presented    $R$-module in $\X$ is projective if and only if every   $R$-module in $\X$ is  $n$-$\X$-injective if and only if every right $R$-module in $\X$ is  $n$-$\X$-flat. This  is a generalization of \cite[Proposition 3.11]{ZL}. Notice that,  when  $n=1$, then $R$ is gr-regular if and only if $1$-$\X$-regular, see \cite{Stenst2}.

The following example show that, for some of class $\X$ of $R$-modules and any $m> n$,  every Gorenstein $n$-$\X$-injective (resp., flat) module is Gorenstein $m$-$\X$-injective. But, Gorenstein $m$-$\X$-injectivity (resp., flatness) does not imply, in general, Gorenstein $n$-$\X$-injectivity (resp., flatness).

 \begin{exs}\label{2po}
 (1) Let $R$ be a graded ring and  $\X$ a class of graded  $R$-module. Then, for any $m> n$, every Gorenstein $n$-$\X$-injective (resp., flat) module is Gorenstein $m$-$\X$-injective (resp., flat), since by \cite[Remark 3.5]{NG}, every $n$-$\X$-injective (resp., flat) module is $m$-$\X$-injective (resp., flat).
  
 (2) Let $R=k[X]$, where $k$ is a field, and $\X$ a class of graded  $R$-modules. Then,   by Remark \ref{2}, every graded left (resp., right) $R$-module is Gorenstein $2$-$\X$-injective (resp., flat)  since every $2$-presented graded  $R$-module is projective. We claim that there is a graded left (resp., right) $R$-module $N$  so that $N$ is not Gorenstein $1$-$\X$-injective (resp., flat). Suppose to the contrary, every graded  left (resp., right) $R$-module  is Gorenstein $1$-$\X$-injective (resp. flat).  If $U$ is a  finitely presented graded   module, then the special  exact sequence 
$0\rightarrow L\rightarrow F_0\rightarrow U\rightarrow 0$ of graded   modules exists. So
 by Proposition \ref{3.1} (resp., Proposition \ref{3.vc}), $U$ is projective and it follows that $R$ is $1$- $\X$-regular or $\X$-regular, contradiction, see \cite[Example 3.6]{NG}.
 \end{exs}
 
\begin{prop}\label{2.1}
Let $\X$ be a class of   $R$-modules. Then, the following statements hold:
\begin{enumerate}
\item [\rm (1)]
If $G$ is a Gorenstein injective   $R$-module, then ${\rm Hom}_{R}(-, G)$ is exact with respect to all special short exact sequences  with modules of  finite projective dimension.
\item [\rm (2)]
If $G$ is a Gorenstein flat right $R$-module, then $G\otimes_{R}-$ is exact with respect to all special short exact sequences with modules of  finite flat dimension.
\end{enumerate}
\end{prop}

\begin{proof}
{(1) Let $0\rightarrow K_n\rightarrow F_{n-1}\rightarrow K_{n-1} \rightarrow 0$ be a special short exact sequence of $U\in\X_n$. It is clear that ${\rm pd}_{R}(U)=m<\infty$ since ${\rm pd}_{R}(K_{n-1})<\infty$.  Also, let  $G$ be  Gorenstein injective.
Then, the following injective resolution of $G$ exists:
$$0\longrightarrow N\longrightarrow A_{m-1}\longrightarrow \cdots \longrightarrow A_{0}\longrightarrow G\longrightarrow 0.$$
So, ${\rm Ext}_{R}^{n+i}(U, A_j)= 0$ for every $0 \leq j \leq m-1$ and any $i\geq 0$. Thus, we deduce that
 ${\rm Ext}_{R}^{n+i}(U,G)\cong{\rm Ext}_{R}^{m+n+i}(U, N)=0$ for any $i\geq 0$. So, ${\rm Ext}_{R}^{1}(K_{n-1},G)\cong{\rm Ext}_{R}^{n}(U,G)=0. $
 
(2) The proof is similar to the one above.}
\end{proof}

Now we can state the main result of this section. 

\begin{thm}\label{3.83} 
Let $R$ be a left $n$-$\X$-coherent ring and $\X$ be a class of   $R$-modules. Then, the following statements are equivalent:
\begin{enumerate}
\item [\rm (1)]
$R$ is self left $n$-$\X$-injective;
\item [\rm (2)]
Every Gorenstein $n$-$\X$-flat   $R$-module is Gorenstein $n$-$\X$-injective;
\item [\rm (3)]
Every Gorenstein flat   $R$-module is Gorenstein $n$-$\X$-injective;
\item [\rm (4)]
Every flat   $R$-module is Gorenstein $n$-$\X$-injective;
\item [\rm (5)]
Every Gorenstein projective   $R$-module is Gorenstein $n$-$\X$-injective;
\item [\rm (6)]
Every projective   $R$-module is Gorenstein $n$-$\X$-injective;
\item [\rm (7)]
Every Gorenstein injective right $R$-module is Gorenstein $n$-$\X$-flat;
\item [\rm (8)]
Every injective right $R$-module is Gorenstein $n$-$\X$-flat;
\item [\rm (9)]
Every Gorenstein $1$-$\X$-injective right $R$-module is Gorenstein $n$-$\X$-flat;
\item [\rm (10)]
Every Gorenstein $n$-$\X$-injective right $R$-module is Gorenstein $n$-$\X$-flat.
\end{enumerate}
\end{thm}
\begin{proof}
 $(1)\Longrightarrow (2)$, $(1)\Longrightarrow (3)$, $(1)\Longrightarrow (4)$, $(1)\Longrightarrow (5)$ and $(1)\Longrightarrow (6)$ follow immediately from Theorem $\ref{3.2}$. 

$(3)\Longrightarrow (4)$, $(4)\Longrightarrow (6)$ and $(5)\Longrightarrow (6)$ are trivial.

$(3)\Longrightarrow (1)$ Assume that $G$ is a projective   $R$-module. Then, 
$G$ is 
flat and so $G$ is Gorenstein $n$-$\X$-injective by (3). So, similar to the proof of ($\Longrightarrow $) of  Proposition $\ref{3.1}$, $G$ is  $n$-$\X$-injective. Thus, the assertion follows from Theorem $\ref{3.2}$.

$(6)\Longrightarrow (1)$ This is similar to the proof of $(3)\Longrightarrow (1)$.

$(1)\Longrightarrow (9)$  By Theorem \ref{3.8bv}, every $1$-$\X$-injective right $R$-module is $n$-$\X$-flat. Suppose that $G$ is Gorenstein $1$-$\X$-injective. So, an  exact sequence 
 $${\mathbf{M}}=\cdots \longrightarrow M_{1}\longrightarrow M_{0}\longrightarrow M^{0}\longrightarrow M^{1}\longrightarrow\cdots,$$
 of $n$-$\X$-flat right $R$-modules exists, where $G={\rm ker}(M^0\rightarrow M^1)$. Let $K_{n-1}$ be  special $\X$-presented  with ${\rm f.d(K_{n-1})}<\infty$. Then, similar to the  proof of Theorem \ref{2.3}(1),  ${\mathbf{M}}\otimes_{R}K_{n-1}$ is exact, and hence $G$ is  Gorenstein $n$-$\X$-flat.

$(9)\Longrightarrow (7)$ By Remark \ref{21}, every injective right $R$-module is $1$-$\X$-injective. So, if $G$ is Gorenstein injective, then an exact sequence  
 $${\mathbf{E}}=\cdots \longrightarrow E_{1}\longrightarrow E_{0}\longrightarrow E^{0}\longrightarrow E^{1}\longrightarrow\cdots$$
 of $1$-$\X$-injective right $R$-modules exists, where $G={\rm ker}(E^0\rightarrow E^1)$. So, if $U\in\X_1$ with ${\rm pd}(U)<\infty$, then $U$ is special $\X$-presented and by Proposition \ref{2.1}, ${\rm Hom}_{R}(U, {\mathbf{E}})$ is exact. Therefore, $G$ is Gorenstein $1$-$\X$-injective.  

$(7)\Longrightarrow (8)$
is trivial since every injective $R$-module is Gorenstein injective.

$(8)\Longrightarrow (1)$
Let $M$ be an injective right $R$-module. Since $M$ is Gorenstein $n$-$\X$-flat, we have an exact sequence:
 $${\mathbf{M}}=\cdots \longrightarrow M_{1}\longrightarrow M_{0}\longrightarrow M^{0}\longrightarrow M^{1}\longrightarrow\cdots,$$ 
where any $M_i$ is $n$-$\X$-flat and $M={\rm ker}(M^0\rightarrow M^1)$. Then, the split exact sequence $ 0\rightarrow M\rightarrow M^{0}\rightarrow L\rightarrow 0$ implies that $M$ is $n$-$\X$-flat, and hence by Theorem \ref{3.8bv}, we deduce that $R$ is self left  $n$-$\X$-injective.

$(1)\Longrightarrow (10)$
 Suppose that $G$ is a Gorenstein $n$-$\X$-injective right $R$-module. By Theorem \ref{3.8bv}(6), every $n$-$\X$-injective right $R$-module is $n$-$\X$-flat. Thus, an  exact sequence
 $${\mathbf{N}}=\cdots \longrightarrow N_{1}\longrightarrow N_{0}\longrightarrow N^{0}\longrightarrow N^{1}\longrightarrow\cdots$$
 of $n$-$\X$-flat right $R$-modules exists, where $G={\rm ker}(N^0\rightarrow N^1)$. Then, similar to the proof of Theorem \ref{2.3}(1),  (10) follows.
 
$(10)\Longrightarrow (7)$ is clear.
\end{proof}

\noindent\textbf{Acknowledgment.}  
The authors would like to thank   the referee for the   very helpful comments and suggestions.  Arij Benkhadra's research reported in this publication was supported by a scholarship from the Graduate Research Assistantships in Developing Countries Program of the Commission for Developing Countries of the International Mathematical Union.

\end{document}